\documentclass[12]{article}
\usepackage{amsmath, amsthm, amssymb}
\usepackage{graphicx}
\usepackage{subfigure}
\usepackage{enumerate}
\usepackage{cite}
\usepackage{ifpdf}

\usepackage{color}
\usepackage{mathrsfs}
\renewcommand{\a}{{\bf a}}
\renewcommand{\b}{{\bf b}}
\renewcommand{\c}{{\bf c}}
\renewcommand{\d}{{\bf d}}
\newcommand{\e}{{\bf e }}
\renewcommand{\t}{{\bf t}}
\renewcommand{\v}{{\bf v}}

\newcommand{\x}{{\bf x }}
\newcommand{\y}{{\bf y}}
\newcommand{\z}{{\bf z }}
\newcommand{\0}{{\bf 0}}
\newcommand{\R}{\mathbb{R}}
\newcommand{\N}{\mathbb{N}}

\newcommand{\Z}{\mathbb{Z}}
\newcommand{\Q}{\mathbb{Q}}
\DeclareMathOperator{\vol}{vol}
\DeclareMathOperator{\conv}{conv}
\newtheorem{theorem}{Theorem}[section]
\newtheorem{corollary}[theorem]{Corollary}
\newtheorem{lemma}[theorem]{Lemma}
\theoremstyle{definition}
\newtheorem{remark}[theorem]{Remark}
\theoremstyle{definition}
\newtheorem{definition}[theorem]{Definition}
\usepackage[normalem]{ulem}
\begin{document}

\title{A new Lenstra-type Algorithm for Quasiconvex Polynomial Integer Minimization with Complexity $2^{O(n\log n)}$\\ }
\date{\today}
\author{Robert Hildebrand and Matthias K\"oppe}
\maketitle
\begin{abstract}
We study the integer minimization of a quasiconvex polynomial with quasiconvex polynomial constraints.  We propose a new algorithm that is an improvement upon the best known algorithm due to Heinz (\textit{Journal of Complexity}, 2005). This improvement is achieved by applying a new modern Lenstra-type algorithm, finding optimal ellipsoid roundings, and considering sparse encodings of polynomials.  For the bounded case, our algorithm attains a time-complexity of $s (r l M d)^{O(1)} 2^{2n\log_2(n) + O(n)}$ when $M$ is a bound on the number of monomials in each polynomial and $r$ is the binary encoding length of a bound on the feasible region. In the general case, $s l^{O(1)} d^{O(n)} 2^{2n\log_2(n) +O(n)}$. In each we assume $d\geq 2$ is a bound on the total degree of the polynomials and $l$ bounds the maximum binary encoding size of the input. 
\end{abstract}

\section{Introduction}
We study the integer minimization of a quasiconvex polynomial with quasiconvex polynomial constraints. That is, given $\hat{F}, F_1, \dots , F_s \in \Z[\x] = \Z[x_1, \dots, x_n]$ quasiconvex polynomials with integer coefficients, we wish to solve the following problem
\begin{equation}
\label{minimization}
\begin{array}{ll}
\min & \hat{F}(\x)\\
\text{subject to} & F_i(\x) < 0 \text{ for all } i = 1, \dots , s\\
& \x \in \Z^n.
\end{array}
\end{equation}
A function $F\colon \R^n \to \R$ is called \textit{quasiconvex} if for every $\alpha \in \R$, the lower level set $\{\x \in \R^n \colon F(\x) \leq \alpha\}$ is a convex subset of $\R^n$.   Some quasiconvex programs reduce nicely to convex programs, see for instance, \cite{boyd2004}, but this is not likely to be the case in general.  Studying quasiconvex integer minimization opens up a larger class of functions that we can optimize over.

We approach the optimization problem by setting $F_0 = \hat{F} - z^*$ and solving the feasibility problem over $Y \cap \Z^n$, where
\begin{equation}
\label{Y}
Y := \left\{ \x \in \R^n : F_i(\x) < 0 , \ i = 0, 1,\dots,s \right\},
\end{equation}
and applying binary search on objective values until we find an optimal solution. Strict inequalities are used to ensure that if $Y$ is non-empty, then it is full dimensional in $\R^n$. Since $F_i(\x) \in \Z$ for all $\x \in \Z^n$, problem \eqref{minimization} can be easily formulated by weak inequalities. This follows from the observation that the inequalities $z < 0$ and $z + 1\leq 0$ are equivalent for $z \in \Z$.

We use a modern Lenstra-type algorithm to solve the integer feasibility problem.
Lenstra's algorithm was the first algorithm to solve integer linear optimization in polynomial time when the dimension is fixed.  It can be applied to any  family of convex sets $\mathcal{C}$ in $\R^n$ provided that we can solve the ellipsoid rounding problem over sets in $\mathcal{C}$.  Khachiyan and Porkolab \cite{khach2000} showed that Lenstra's algorithm could be generalized to operate on convex semialgebraic sets, having time-complexity of $l^{O(1)}(sd)^{O(n^4)}$. For the specific case of quasiconvex polynomial minimization, the current best algorithm is due to Heinz and has time-complexity of $s l^{O(1)} d^{O(n)} 2^{O(n^3)}$, where $d \geq 2$ is an upper bound on the total degree of the polynomials and $l$ is the maximum binary encoding size of all coefficients.   

Our improvement over Heinz's algorithm comes primarily from the modern Lenstra-type algorithm that we present.  Heinz developed a shallow cut separation oracle to show that Lenstra's original algorithm applies to the quasiconvex minimization problem (\ref{minimization}).  We generalize Heinz's shallow cut separation oracle to show that the modern Lenstra-type algorithm works for the quasiconvex minimization problem (\ref{minimization}). We also provide a structure of evaluating polynomials that exploits sparsity, which allows us to state a more precise complexity of the algorithm based on the number of monomials given in the input.

\begin{theorem}
\label{main}
Let $\hat{F}, F_1, \dots, F_s \in \Z[\x]$ be sparsely encoded quasiconvex
polynomials. Let $d\geq 2$ be an upper bound for the degree of the
polynomials $F_0, \dots, F_{s}$, let $M$ be the maximum number of monomials
in each, and let the binary length of the coefficients be bounded by $l$.
Then there exists an algorithm for the minimization problem
\eqref{minimization} which computes a minimum point or confirms that such a
point does not exist. 
\begin{enumerate}[\rm(a)]
\item 
If the continuous relaxation of the feasible region is bounded such
that $r$ is the binary encoding length of a bound on that region with $r
\leq l d^{O(n)}$, then the algorithm has time-complexity of $s (r l M
d)^{O(1)} 2^{2n\log_2(n) + O(n)}$ and output-complexity of
$(l+r)(dn)^{O(1)}$.
\item Otherwise, the algorithm has time-complexity of $s
l^{O(1)} d^{O(n)} 2^{2n\log_2(n) }$ and output-complexity of $l d^{O(n)}$.
\end{enumerate}
\end{theorem}
\indent For $d = O(1)$, this complexity is $s l^{O(1)} 2^{2n\log_2(n) + O(n)}$.\\
\indent If $d = O(n^k)$ for some $k >0$, then the complexity becomes $s l^{O(1)} 2^{O(n \log(n))}$.\\

Lenstra's algorithm solves the integer feasibility problem for a convex set $Y$ by first finding a pair of concentric ellipsoids, $E, E' = \frac{1}{\beta} E$ such that $E' \subseteq Y\subseteq E$, where $E'$ is a scaled version of $E$ with respect to the center. If we can determine that $E' \cap \Z^n$ is non-empty, then we are done. Otherwise, we find a direction of minimal width and branch into integer hyperplanes
that cover $E$, creating lower dimensional subproblems to solve. The same approach is then applied to solving each lower dimensional subproblem. \\

The complexity of Lenstra's algorithm is guided primarily by the number of subcases that it must evaluate. As shown in section $4$, the number of subcases in each step is bounded by $2 \omega(n) \beta(n) + 3$ where $\omega(n)$ is called the lattice width direction of the inner ellipsoid and $\beta(n)$ is the rounding radius or scaling between the concentric ellipsoids that are obtained. The lattice width $\omega(n)$ will be used to determine if $E' \cap \Z^n \neq \emptyset$. This leads to a total number of subcases bounded by 
$$
\text{total \# of subcases}\ \leq \ \prod_{i=1}^n 2 \omega(i) \beta(i) + 3.
$$ 
\begin{figure}
\begin{center}
\begin{picture}(0,0)%
\includegraphics{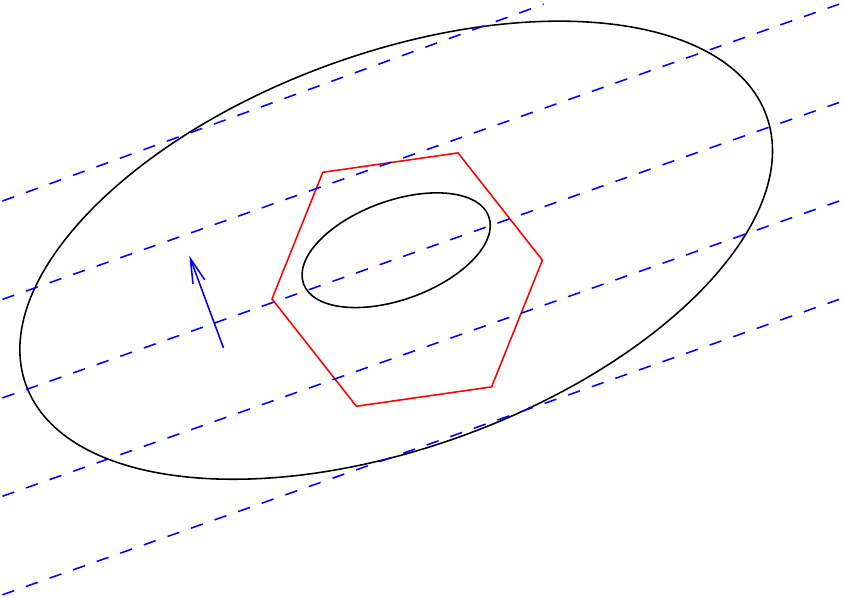}%
\end{picture}%
%
%
\setlength{\unitlength}{3947sp}%
\begingroup\makeatletter\ifx\SetFigFont\undefined%
\gdef\SetFigFont#1#2#3#4#5{%
  \reset@font\fontsize{#1}{#2pt}%
  \fontfamily{#3}\fontseries{#4}\fontshape{#5}%
  \selectfont}%
\fi\endgroup%
\begin{picture}(4040,2859)(229,-2248)
\put(2101,-501){\makebox(0,0)[lb]{\smash{{\SetFigFont{12}{14.4}{\rmdefault}{\mddefault}{\updefault}{\color[rgb]{0,0,0}$E'$}%
}}}}
\put(2691,-251){\makebox(0,0)[lb]{\smash{{\SetFigFont{12}{14.4}{\rmdefault}{\mddefault}{\updefault}{\color[rgb]{1,0,0}$Y$}%
}}}}
\put(3381, 19){\makebox(0,0)[lb]{\smash{{\SetFigFont{12}{14.4}{\rmdefault}{\mddefault}{\updefault}{\color[rgb]{0,0,0}$E$}%
}}}}
\end{picture}%

\caption{Main idea of Lenstra's Algorithm}
\end{center} 
\end{figure}
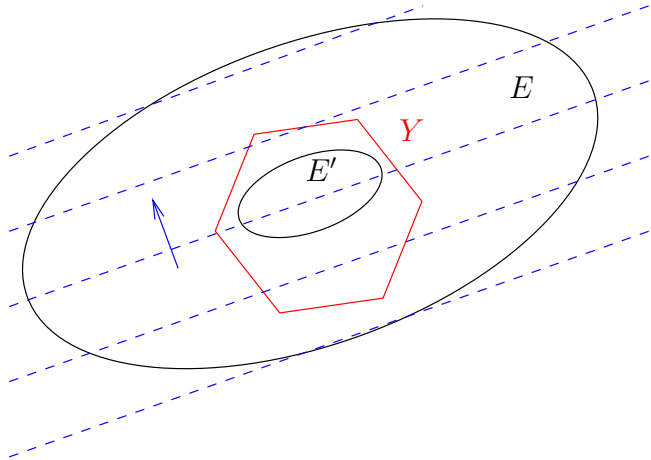

The idea of the modern Lenstra-type algorithm is not new.  Kannan was the first to reduce the number of subcases by solving lattice problems.   Recently, extraordinary new lattice algorithms and bounds from the geometry of numbers reveal a better complexity.  A new idea we present is finding an ellipsoid rounding with an optimal rounding radius of $\beta(n) = O(n)$.   This is done by using results for the maximal radius of a ball inscribed in the convex hull of $m$ points on the sphere of radius 1.  We will explain these improvements and how they affect the complexity of our algorithm.

In section 2 we will discuss ellipsoid rounding from the point of view of the shallow cut ellipsoid method and show how we improve $\beta(i)$ from $O(n^{3/2})$ to $O(n)$, which has never been done before in Lenstra's algorithm. This improvement allows for a better exponential coefficient in the resulting complexity.

In section 3 we will explain how Kannan's improvement of Lenstra's algorithm reduces the number of subcases exponentially. This section will begin with a discussion of lattice theory, flatness directions, and the geometry of numbers, and we will reveal that $\omega(n)$ can be improved from $2^{O(n^2)}$ as used in the original Lenstra algorithm \cite{lenstra83}  along with Khachiyan and Porkolab\cite{khach2000} and in Heinz \cite{heinz05}, to $O(n)$ as stated in Eisenbrand \cite{50years}.  The important feature that we point out is how new lattice algorithms allow this computation to be done determistically in $2^{O(n)}$ time, as opposed to $2^{O(n\log n)}$, creating a better overall complexity.
  
In section 4 we state our modern Lenstra-type algorithm more precisely and explain its complexity.

In section 5 we will discuss polynomial encoding and quasiconvex polynomials and then show how to make shallow cuts for sets given by quasiconvex polynomials.  This allows Lenstra's algorithm to be applied to our problem.

In section 6 we discuss the proof of Theorem \ref{main}.

\section{Ellipsoid Rounding}

The first step of Lenstra's algorithm is to find a pair of concentric ellipsoids, one inside and one containing the feasible region. 
We will write ellipsoids in the form $E(A,\a) = \{ \x\in \R^n : ||\x-\a||_{A^{-1}} \leq 1\}$ where $|| \v ||_{B} := \sqrt{\v^T B \v}$, $A \in \R^{n\times n}$ is a positive definite matrix and $\a \in \R^n$.
For example, $E(\alpha^2 I, \0)= B(\alpha,\0)$ is the ball of radius $\alpha$ centered at the origin.  Let $Y$ be a convex set. The ellipsoid $E(A,\a)$ is a \emph{$\beta$-rounding} of $Y$ if 
$$
E\bigl(\tfrac{1}{\beta^2} A,\a\bigr) \subseteq Y \subseteq E(A,\a)
$$
where $\beta$ is called the \emph{radius} of the rounding\cite{nesterov07}. John~\cite{john48} showed there exists a $n$-rounding for any convex set. Conversely, for a simplex, a $n$-rounding is the best possible rounding. Finding an optimal rounding will reduce the number of subcases that need to be analyzed in Lenstra's algorithm. In section 4 we will show precisely where $\beta$ affects the complexity of integer optimization. In this section we will explain how to construct a $O(n)$-rounding. 

There are several methods for ellipsoid rounding. Nesterov describes an algorithm to obtain a $\gamma n$-rounding, $\gamma > 1$ for an arbitrary convex set and also how to obtain a $\gamma \sqrt{n}$-rounding for centrally symmetric convex sets~\cite{nesterov07}, but each is based on the assumption that a difficult optimization problem can be solved. For a specific case, Nesterov uses linear programming, whereas we would need to maximize over nonlinear polynomials. In our model, no supplementary optimization problem need be solved. Ellipsoid roundings have also been studied by Khachiyan~\cite{khach96}, which was improved by~\cite{kumar05} and~\cite{todd07}. Some other methods use a volumetric barrier~\cite{anstre99, anstre02}.
Unfortunately, none of these have been shown to round general convex sets.

We will use the original approach, which is to employ the \emph{shallow cut ellipsoid method}\cite{schriver88}. This can be applied to any class of convex sets $\mathcal C$ for which there exists a \textit{shallow cut separation oracle}.
\begin{definition}[Shallow cut separation oracle \cite{schriver88}]
A \textit{shallow cut separation oracle} for a convex set $Y\subset \R^n$ is an oracle which, for an input $\a\in \Q^n$ and a rational positive definite matrix $A$, outputs one of the following:
\begin{enumerate}
\item verification that $E(A,\a)$ is a $\beta$-rounding of $Y$ or,
\item a vector $\c \in \Q^n$ such that the half-space 
$$\left\{\x\in \R^n \colon \c^T\x \leq \c^T \a + \frac{1}{n+1} \sqrt{\c^T A^{-1}\c}\right\} \supseteq Y \cap E(A,\a).$$
\end{enumerate}
\end{definition}
\begin{theorem}[Shallow Cut Ellipsoid Method\cite{schriver88}]
There exists an oracle-polynomial time algorithm that for any number $\epsilon > 0$ and for any circumscribed closed convex set $Y \subset B(R,\0)$ given a shallow cut separation oracle finds a positive definite matrix $A \in \Q^{n\times n}$ and a point $\a \in \Q^n$ such that one of the following holds:\\
(i) $E(A,\a)$ is a $\beta$-rounding of $Y$.\\
(ii) $Y\subset E(A,\a)$ and $\text{vol}(E(A,\a)) \leq \epsilon$\\
This algorithm runs in time oracle-polynomial in $n + \langle R \rangle + \langle \epsilon \rangle$.\\
\end{theorem}
The main difficulty in generalizing Lenstra's algorithm to different classes of convex sets is creating shallow cuts.  Khachiyan and Porkolab show that for an arbitrary semialgebraic set, a shallow cut can be computed in $l^{O(1)} (s d)^{O(n^3)}$ time\cite[Lemma 4.1]{khach2000}.  In our algorithm, we intend to do much better for the specific case of quasiconvex polynomials, although this complexity discussion will be saved for section 5 when we discuss quasiconvex polynomials. 

Suppose that we have an ellipsoid $E(A,\a) \supset Y$. Ellipsoids are affine transformations of the unit ball; that is, there exists an affine map $\tau:\R^n \to \R^n$ such that $\tau(E(A,\a)) = B(1,\0)\supset \tau(Y)$. The standard method of creating a shallow cut separation oracle is to observe that points in $B( \frac{1}{n+1},\0)\setminus\tau(Y)$ will often admit a shallow cut. For the case when $Y$ is a polyhedron, finding such a point directly admits a shallow cut, whereas in the case of quasiconvex minimization, Heinz showed that with a little more work, we could find a shallow cut. This will be explained in detail in section \ref{quasiconvexMinimization}.
  The remainder of this section will focus on realizing a $\beta$-rounding. 

On the other hand, if we find a set of points $V$ within the ball of radius $\frac{1}{n+1}$ that are in $\tau(Y)$, then any ball $B(0,\frac{1}{\beta}) \subset \conv(V)$ will admit a rounding since $\tau^{-1}(B(\frac{1}{\beta},\0)) = E(\frac{1}{\beta^2}A, \a) \subset Y \subset E(A,\a)$. The rounding radius $\beta$ is then dependent on the maximum inscribed sphere inside $\conv(V)$, as seen in Figure \ref{radius}.\\

Using the cross-polytope $$ V_{L} = \{\pm \frac{1}{n+1} \e_i : i=1, \dots , n\},$$ where $\e_i$ is the $i^{\mathrm{th}}$ unit vector,~\cite{schriver88} obtains a $O(n^{3/2})$-rounding of a polytope, which is also used in \cite{khach2000}.
Heinz used this idea to obtain a rounding of a convex region given by quasiconvex polynomials~\cite{heinz05}.
In order to overcome numerical issues of requiring exact arithmetic, Heinz chose the points $$V_H= \{\pm \lambda_i \e_i: i=1, \dots, n\}$$ where $\frac{1}{n+3/2} < \lambda_i < \frac{1}{n+1}$. Heinz's choice also obtains an $O(n^{3/2})$-rounding. \begin{figure}
\begin{center}
\input{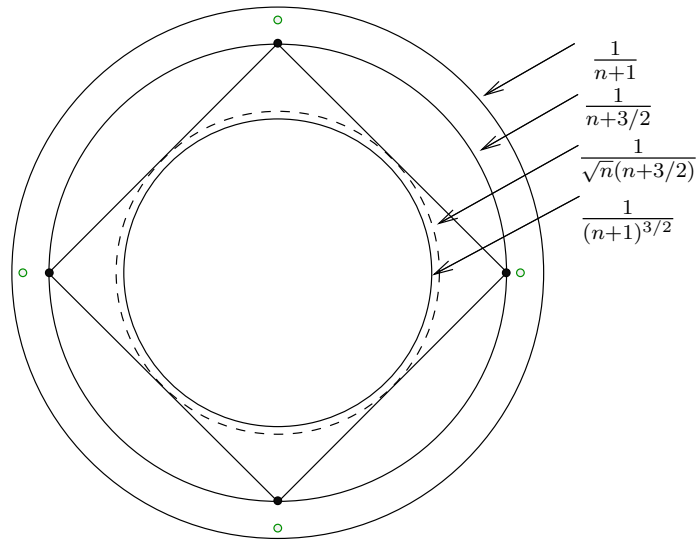}
\caption{Under the affine transformation, the open dots are test points for Heinz's separation oracle. This method avoids the numerical issues of finding exact square roots.}
\end{center}
\end{figure} 
We will generalize and improve Heinz's method by applying sphere approximating polytopes of Kochol~\cite{kochol94} that attain an optimal bound within a constant factor. A note from Kochol, modified to give more detail, shows the following result.
\begin{figure}
\begin{center}
\label{radius}
\ifpdf
\input{radiuspdf.tex}
\else
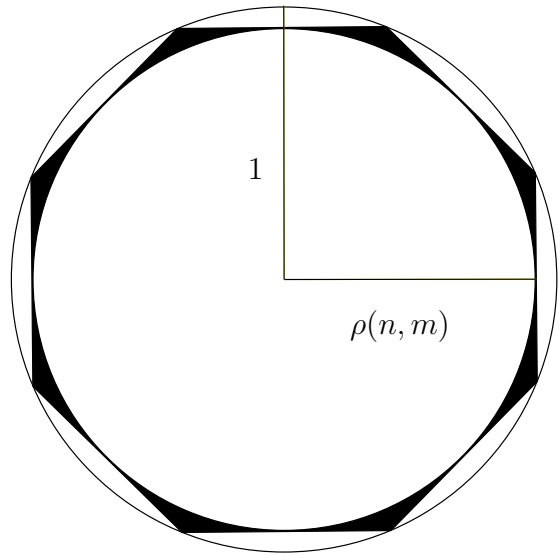
\fi
\caption{The maximal radius $\rho(n,m)$ of a ball (with center at the origin) contained in the convex hull of $m$ points chosen from the $n$-dimensional sphere of radius $1$. }
\end{center}
\end{figure}
\begin{figure}
\begin{center}
\ifpdf
\input{approxmodelpdf.tex}
\else
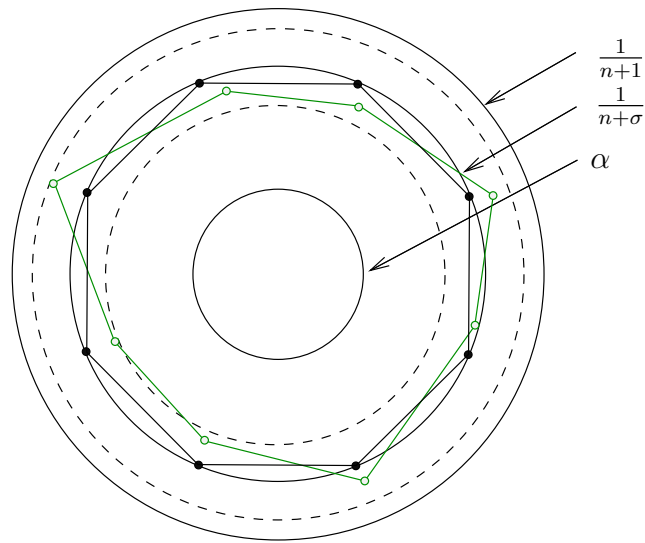
\fi\\
\caption{An inscribed polytope and a rational approximation with vertices, respectively, as filled in dots and open dots. The innermost circle is the ball that we can guarantee will remain inside the convex hull of the approximated points.}
\end{center}
\end{figure}
\begin{theorem}[Theorem 3 in~\cite{kochol04}]
\label{rounding}
Let $n,m$ be positive integers, $2n\leq m \leq c^n$, where $c > 1$ is a constant. Let $\rho(n,m)$ as the maximal radius of a ball (with center at the origin) contained in the convex hull of $m$ points chosen from the $n$-dimensional sphere of radius $1$. Then there exist constants $c_1$ and $c_2$ such that 
$$
c_1 \sqrt{\frac{\log(m/n)}{n}} \leq \rho(n,m) \leq c_2 \sqrt{\frac{\log(m/n)}{n}}.
$$
Furthermore, there exists a polynomial time algorithm in $n$ and $m$ to construct a set of vectors $V \subset \Z^n$ with $|V| \leq m$ such that the polytope with extreme points $\{ \v/||\v||_2 : \v \in V\}$ is symmetric across all axes and attains such bounds.
\end{theorem}
Kochol notes that choosing $m: = n^2$ points improves the $O(n^{3/2})$-rounding to $O(n^{3/2} / \sqrt{\log n})$ and still allows a polynomial time rounding, and improves upon the complexity for Lenstra's algorithm given in~\cite{schriver88}. Theorem \ref{rounding} demonstrates that using an exponential number of points is necessary to obtain an $O(n)$-rounding via the shallow cut ellipsoid method. A tighter rounding is  advantageous, even if it requires an exponential number of evaluations.  In our new algorithm we choose a single exponential number of points, $m:=n2^n$ will suffice, to obtain an $O(n)$-rounding. This improves the exponential coefficient in the final complexity.\\

We will now show that numerical approximations of Kochol's approximating spheres will still allow for an optimal rounding. We denote the sphere of radius $r$ as $S^{n-1}(r) = \{ \x \in \R^{n} \colon ||\x||_2 =r\}$. A \emph{$1$-net} of $S^{n-1}(r)$ is a set of points $N\subset S^{n-1}(r)$ such that for any point $\z \in S^{n-1}(r)$, there exists a point $\v$ such that $||\v- \z||_2 \leq 1$.
\begin{lemma}
\label{netapprox}
Let $N$ be a $1$-net of $S^{n-1}(1)$, let $0 \leq \alpha < \frac{1}{2}$ and let $0 < \epsilon < \frac{1}{2} - \alpha $. Suppose that $\tilde{N}$ is an \emph{$\epsilon$-approximation} of $N$, that is to say that for all $\v \in N$ there exists a $\tilde{\v} \in \tilde{N}$ such that
$$
||\v - \tilde{\v}||_2 \leq \epsilon,
$$
then  $\conv(\tilde{N}) \supseteq S^{n-1}(\alpha)$.
\end{lemma}
\begin{proof}
Suppose there exists a point $\z \in S^{n-1}(\alpha)$ that does not belong to $\conv(\tilde{N})$, then separating $\z$ from $\conv(\tilde{N})$ by a hyperplane $p_{\z}$ we get a cap of $S^{n-1}(1)$ which is disjoint from $\tilde{N}$ and its top $\t$ where $\t$ is perpendicular to~$p_{\z}$. Since $\t$ is in $S^{n-1}(1)$, there exists a point $\v \in N$ from the 1-net that satisfies $||\t - \v||_2 \leq 1$. See Figure \ref{geometry} for the geometry.
\begin{figure}
\label{geometry}
\begin{center}
\ifpdf
\input{geometrypdf.tex}
\else
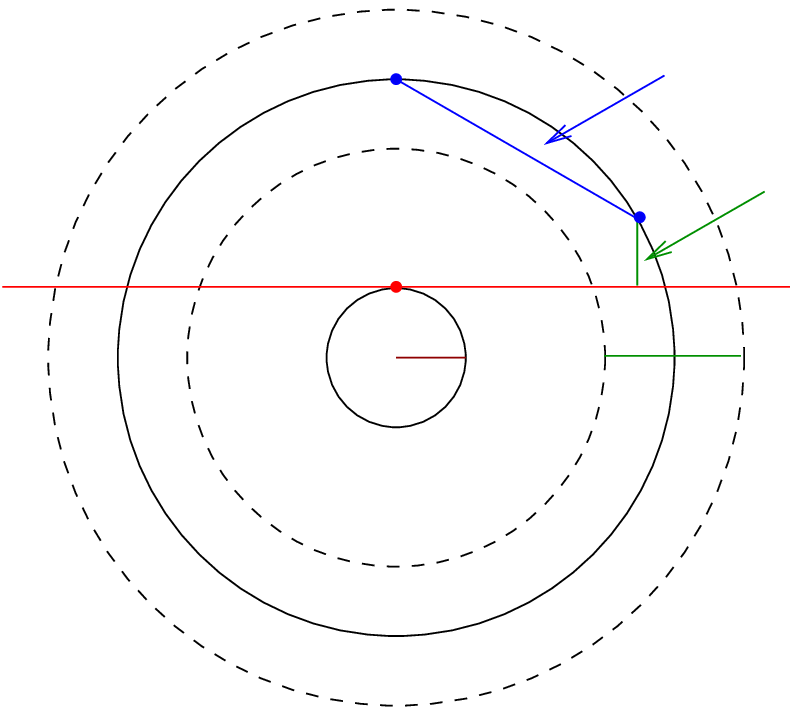
\fi
\caption{Geometry of the proof for Lemma \ref{netapprox}}
\end{center}
\end{figure}
Letting $d$ be the minimum distance between $\v$ and the hyperplane $p_{\z}$, we can see that $d \geq \frac{1}{2} - \alpha > \epsilon$, which is a contradiction since $\tilde{N}$ is an $\epsilon$-approximation of $N$.
\end{proof}

\begin{corollary}
\label{rationalapprox}
Let $N$ be the set of points given by Kochol's construction for an approximation of the unit sphere and let $\tilde{N}$ be an $\epsilon$-approximation of $N$ with $0 < \epsilon < \frac{1}{4}$. Then there exist a constant $c_1(\epsilon)$ such that the ball of radius $c_1 \sqrt{\frac{\log(m/n)}{n}}$ is contained in $\conv(\tilde{N})$.
\end{corollary}
\begin{proof}
Using Lemma \ref{netapprox} with $\alpha = \tfrac{1}{4}$, the proof is very similar to Theorem 1 in~\cite{kochol94}.
\end{proof}
This is now the template for our separation oracle.  We will choose $m:=n2^n$ test points according to an approximation of Kochol's set of points.  If all the test points are feasible, we obtain an $O(n)$-rounding.  Otherwise, we find an infeasible test point and generate a shallow cut.  The specific algorithm for finding a shallow cut for quasiconvex polynomials will be presented in section 5.
\section{Integer Feasibility and Subcases}
We assume now to have an ellipsoid rounding 
$$E(\tfrac{1}{\beta^2}A,\a) \subseteq Y \subseteq E(A,\a).$$
 The next step in Lenstra's algorithm is to determine if the inner ellipsoid contains an integer point. A simple, yet powerful, way to do that is \emph{Khinchin's Flatness Theorem}, which roughly states that if the minimum width of a convex body is greater than some constant $\omega(n)$, then the convex body contains an integer point. If the minimum width is less than $\omega(n)$, then we can branch into integer hyperplanes perpendicular to the minimum width direction. Since we must branch on the larger ellipsoid, we will have fewer than $\beta (n)\omega(n) $ subcases to branch into. We will first review lattice theory and flatness directions, and present theorems for reducing the complexity of Lenstra's algorithm.

\begin{figure}
\begin{center}
\ifpdf
\input{LenstraAlgWidthPlanespdf.tex}
\else
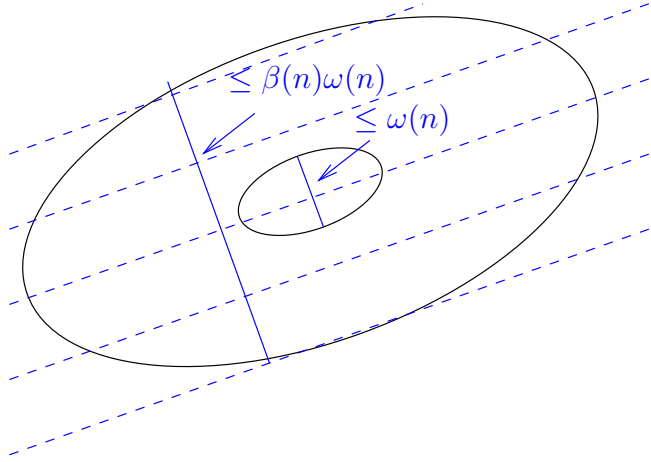
\fi\\
\end{center}
\caption{If the lattice width of the inner ellipsoid is greater than $\omega(n)$, then we know it contains a lattice point.  Otherwise, we are in the case displayed here, with roughly fewer than $\beta(n)\omega(n)$ subcases to project into.  We later use the loose bound of $2 \beta(n) \omega(n) + 3$ for the number of subcases, as shown later in this section.}
\end{figure}

\subsection{Lattices}
Given $m$ linearly independent vectors $\b_1, \b_2, \dots, \b_m \in \R^n$, the \emph{lattice} $\Lambda$ generated by $B = [\b_1, \dots, \b_m]$ is the linear transformation of $\Z^n$
$$
\Lambda = \mathcal{L}(B) = \left\{B\z : \z \in \Z^n \right\}.
$$
\noindent The set of vectors $\b_1, \dots, \b_m$, or similarly $B $, is called a \emph{basis} for the lattice $\Lambda$.
\\
Lattices naturally arise when looking for integer points in ellipsoids, since an ellipsoid is an affine transformation of $B(1,\0)$.  We will need the following related properties of a lattice.
The \emph{dual lattice} $\Lambda^*$ is given by
$$
\Lambda^* = \{ \v \in \text{span}(B) : \v^TB\in \Z^m \ \}.
$$
In particular, one can show that when $B$ is full rank, $\Lambda^* = \mathcal{L}((B^{-1})^T)$. We will use a transference bound later, which is an inequality relating the properties of a lattice and its dual.

The \textit{covering radius} $\mu(\Lambda)$ is the smallest number
$\alpha>0$ such that the set of closed balls of radius $\alpha$ centered at each lattice point completely covers all of $\R^n$. 

The shortest vector problem (SVP) is a well-studied and important problem in lattice theory with many applications.  The goal  is to find a non-zero lattice vector $\x \in \Lambda\setminus\{\0\}$ 
with minimal length $\lambda(\Lambda)$. For our purposes, we will find the shortest vector with respect to the Euclidean norm; therefore,
$$
\lambda(\Lambda) = \min_{\v \in \Lambda\setminus \{\0\}} ||\v||_2 = \min \left\{||\v||_2 : \v = B\z ,\ \z \in \Z^n\setminus\{\0\}\right\}.
$$ 

Similarly, the closest vector problem (CVP) is to find the closest vector in the lattice to a given point $\t$.  Again, we will use the Euclidean norm for this problem.  We will define CVP$(\Lambda,\t)$ as 
$$
\text{CVP}(\Lambda, \t) = \arg\min_{\v \in \Lambda} ||\v - \t||_2. 
$$
For further review on lattices, see, for instance,~\cite{bert} or \cite{50years}.

\subsection{Lattice Widths and the Shortest Vector Problem}
Kannan first observed that SVP could be used to minimize the number of branching directions in Lenstra's algorithm~\cite{kannan83}. We follow Eisenbrand in presenting this in the context of flatness directions~\cite{50years}.
Let $K\subset \R^n$ be a non-empty closed subset of $\R^n$ and let $\d \in \R^n$ be a vector. The \emph{width of $Y$ along $\d$} is the number
$$
w_{\d}(K) = \max\{\d^T \x : \x \in K\} - \min\{\d^T \x :\x \in K\}.
$$
The \textit{lattice width} of $Y$ is defined as 
$$
w(K) = \min_{\d \in \Z^n\setminus \{\0\}} w_{\d}(K),
$$
and any $\d$ that minimizes $w_{\d}(K)$ is called a \textit{flatness direction} of $K$.
\begin{theorem}[Khinchin's flatness theorem~\cite{khin48}]
There exists a function $\omega(n)$, depending only on the dimension, such that if  $K\subset \R^n$ is convex and $w(K) > \omega(n)$, then $K$ contains an integer point.
\label{flatness}
\end{theorem}

The  best known bound for $\omega(n)$ is $O(n^{3/2})$ and it is conjectured that $\omega(n) = \Theta(n)$ \cite{bana99}. We will see in the next subsection that, for the specific case of ellipsoids, we can obtain this bound.\\

Flatness directions are invariant under dilations. This is easily shown for the case of ellipsoids. 
\begin{lemma}
\label{betaflat}
Let $\d \in \Z^n$ be a flatness direction for $E(A,\a)$. 
Then for any $\beta \in \R$, $\d$ is a flatness direction for $E(\frac{1}{\beta^2} A,\a)$ with
$$
w(E(\tfrac{1}{\beta^2} A, \a))=\tfrac{1}{\beta} w(E(A,\a)).
$$
\end{lemma}
\begin{remark}
\label{flatdir} 
For an ellipsoid, a flatness direction can be computed by solving the shortest vector problem in the lattice $\Lambda = \mathcal{L}((A^{1/2})^T)$. To see this, consider the width along a direction $\d$ of the ellipsoid $E(A,\0)$, 
\begin{align*}
w_\d(E(A,\a))&= \max\{\d^T \x : \x \in E(A,\0)\} - \min\{\d^T \x :\x \in E(A,\0)\}\\
&= \max_{\x_1, \x_2 \in E(A,\0)} \d^T(\x_1-\x_2).
\end{align*}
We have $\d^T (\x_1 - \x_2) = \d^T A^{1/2}(A^{-1/2}\x_1 - A^{-1/2}\x_2)$ where $A^{-1/2}\x_1 $ and $A^{-1/2}\x_2$ are contained in the unit ball if and only if $\x_1, \x_2 \in E(A,\0)$. Thus properly choosing $\x_1$ and $\x_2$ on the boundary of $E(A,\0)$, we see that 
$$w_\d(E(A,\0)) = 2 ||\d^T A^{1/2}||_2 = 2 ||(A^{1/2})^T\d||_2.$$
Finding the minimum lattice width is reduced to solving a SVP over the lattice $\Lambda$.
After this transformation, integer points from $E(\frac{1}{\beta^2} A, \a) \cap \Z^n$ now live in $B(1, A^{-1/2} \a) \cap \Lambda^*$ where $\Lambda^* = \mathcal{L}(A^{-1/2})$.\\
There are two computations arising here:\\
1. We find a shortest vector in $\Lambda$ to determine a flatness direction.  If $2\lambda(\Lambda) \leq \omega(n)$, then we will project into a minimal number of subcases.  \\
2. If $2\lambda(\Lambda) > \omega(n)$ then we have confirmed that $E(\frac{1}{\beta^2}A, \a)$ contains an integer point.   To recover this lattice point, we solve $\z = \mathrm{CVP}(\Lambda^*, A^{-1/2} \a)$ and then set $\x^* = A^{1/2}\z$. Since $\x^*$ is then a closest integer point to $\a$ with respect to the ellipsoidal norm and we know that $E(\frac{1}{\beta^2}A,\a)$ contains an integer point, we have $\x^* \in E(\frac{1}{\beta^2}A,\a) \cap \Z^n$.          
\end{remark}
\begin{figure}
\label{LenstraAlgDetection}
\begin{center}
\ifpdf
\input{LenstraAlgDetectionpdf.tex}
\else
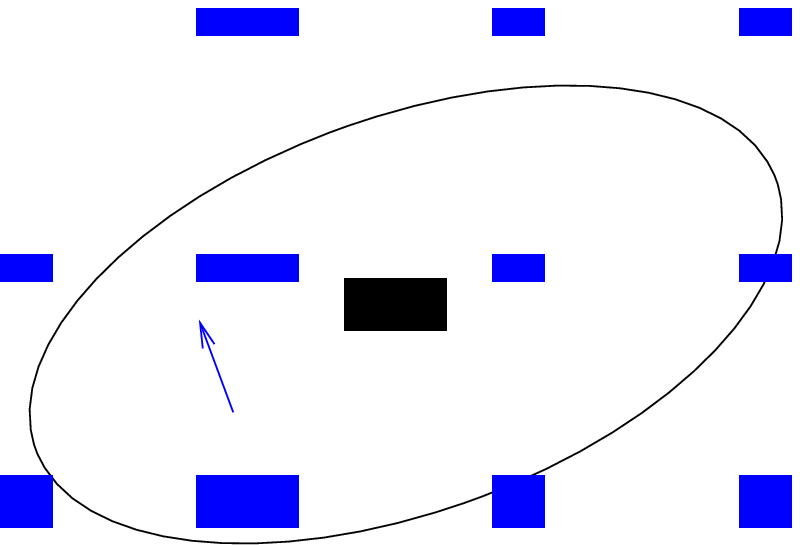
\fi
\caption{We use SVP on $\Lambda = \mathcal{L}((A^{1/2})^T)$ to solve for a flatness direction.  If this shows that the lattice width $w(E(\frac{1}{\beta^2} A, \a)) > \omega(n)$, then it contains a lattice point.  We then recover a feasible integer point by solving CVP$(\Lambda^* = \mathcal{L}(A^{-1/2}), A^{-1/2}\a)$, where the target point is the center transformed by the lattice basis.  The solution transformed back to the original space will be the closest lattice point to $\a$ with respect to the ellipsoidal norm, and therefore is contained in the ellipsoid.}
\end{center}
\end{figure}
\subsection{Results from the Geometry of Numbers}
The geometry of numbers produces a small bound on the lattice width of an ellipsoid not containing an integer point.  Using properties of LLL reduced bases, Lenstra originally observed that this value did not exceed $2^{O(n^2)}$~\cite{lenstra83}. By considering the specific case of ellipsoids, we can achieve an $O(n)$ bound. For an arbitrary lattice, the product of the length of a shortest vector in a lattice and the covering radius of the dual lattice is bounded by a constant $f(n)$ dependent only on dimension. Using the Fourier transform applied to a probability measure on a lattice, Banaszczyk showed that this function is bounded by a linear factor in the dimension $n$.
\begin{theorem}[Theorem 2.2 in~\cite{bana93}] 
\label{cover}
Let $\Lambda \subset \R^n$ be a lattice with $n\geq 1$. 
Then $\lambda(\Lambda) \mu(\Lambda^*) \leq f(n) \leq \frac{1}{2} n$.
\end{theorem}
If we assume that a specific ellipsoid does not contain a lattice point, then the covering radius of the associated lattice is greater than one. Since the lattice width of an ellipsoid is simply twice the length of a shortest vector, we obtain the following inequality for ellipsoids.
\begin{theorem}[Theorem 14.26 in~\cite{50years}]
\label{width}
If $E(A,\a) \subset \R^n$ is an ellipsoid that does not contain an integer point, then $w(E(A,\a)) \leq 2 f(n)$.
\end{theorem}
A  convenient bound follows directly from Theorems \ref{cover} and \ref{width}.
\begin{corollary} 
\label{flatness}
Let $E(A,\a) \subset \R^n$ be an ellipsoid not containing an integer point, then $w(E(A,\a)) \leq n$. 
\end{corollary}
Hence, if $E(\frac{1}{\beta^2} A,\a)$ does not contain an integer point, then $w(E(A,\a)) \leq n \beta$.

\subsection{Complexity of the Shortest Vector Problem and the Closest Vector Problem}
The shortest vector problem (SVP) has been shown to be NP-hard, even to approximate within a constant factor~\cite{miccisvp01}. Until recently, the best known deterministic solution to SVP was by Kannan, with time-complexity $n!= 2^{O(n \log n)}$~\cite{kannaninf}. The well-known Ajtai, Kumar, and Sivakumar~\cite{ajtai01} sieving method is a probabilistic method that solves SVP with very high probability and achieved the first single exponential time-complexity, which was shown by~\cite{nguyen} to be $2^{5.9n}$. Micciancio and Voulgaris improved this type of method to achieve a run time of $2^{3.199 n}$~\cite{micci092}. 

The closest vector problem (CVP) has also been shown to be NP-hard, even to approximate it within a polynomial factor~\cite{dinur03}.  Kannan presented an algorithm to solve CVP in $n!$ time~\cite{kannaninf}, and although there have been some improvements~\cite{hanrot07, helfrich85}, none have been able to achieve a single exponential time-complexity.

Micciancio and Voulgaris discovered the first deterministic single exponential algorithms for CVP and SVP~\cite{micciancio10}.  This is an exciting and impressive result. 
\begin{theorem}[Corollarys 4.3 and 4.4 in \cite{micciancio10}]
\label{SVP}
There are deterministic algorithms to solve SVP and CVP, with the Euclidean norm, that both have time-complexity $n^{O(1)}2^{2n}$.
\end{theorem}

This result, for the first time, allows the complexity of SVP to be much smaller than the  complexity of Lenstra's algorithm.
\subsection{Projection}
We will need the following simple lemma, which can, for instance, be found in \cite{henk97}. We indicate it here with a proof to give a precise complexity. The proof uses the common fact that if $B$ is a basis for a lattice $\Lambda$, and $U$ is a unimodular matrix, then $BU$ is also a basis for $\Lambda$. 
\begin{lemma}
\label{basis}
Suppose $\Lambda \subset \Z^n$ is a lattice with basis $\{ \b_1, \dots, \b_m\}$ and $m \geq 2$. Suppose $\d \in \Lambda\setminus \{\0\}$ is \emph{primitive} 
(i.e., $\alpha \d \notin \Lambda$ for all $0 < \alpha < 1$), and let $\lambda_i \in \Z$ for $i=1, \dots, m$ such that $\d = \lambda_1 \b_1 + \dots + \lambda_m \b_m$. Then there exists an algorithm that computes vectors $\bar{\b}_2, \dots ,\bar{\b}_{m}$ such that $\{\d,\bar{\b}_2, \dots , \bar{\b}_{m}\}$ is a basis for $\Lambda$. This algorithm has time-complexity $(n \log L)^{O(1)}$ where $L$ is an upper bound on the absolute values of $\lambda_i$ and the entries of $\b_i$ for all $i=1, \dots, m$.
\end{lemma}
\begin{proof}
Let $B = [\b_1, \dots, \b_m]$ and let $B' = [\d, \b_2, \dots, \b_m]$. Without loss of generality, we assume $B'$ has rank $m$, otherwise we can simply reorder the basis vectors. Let $A\in \Z^{n\times n}$ such that $B' = B A$. 
We now decompose $A$ into Hermite normal form, which can be done in polynomial time in the input size and the dimension~\cite{kannan79}. That is, we find a unimodular matrix $U\in \Z^{n\times n}$ and an upper triangular matrix $T\in \Z^{n\times m}$ such that $A = U T$; therefore, $B' = (BU) T$. 
There are several algorithms to compute Hermite normal form. For a worst case complexity bound, we use Storjohann and Labahn \cite{sto96} with run time $O(n^3M(nL))$ where $L$ is a bound on the maximum binary encoding length of each entry of $A$, and $M(t)$ is the time required multiply two numbers of size~$\left\lceil t \right\rceil$. The entries of $A$ are all $1$'s, $0$'s, and $\lambda_i$'s. 
Since $U$ is unimodular, $BU = [\bar{\b}_1, \dots ,\bar{\b}_m]$ is a basis for $\Lambda$. Since $T$ is upper triangular, we find that $T_{11} \bar{\b}_1 = \d$, and because $\d$ is primitive and $ \bar{\b}_1\in \Lambda$, we have $\bar{\b}_1 = \d$.
Thus $\{\d, \bar{\b}_2, \dots, \bar{\b}_{m}\} $ is a basis for $\Lambda$.
\end{proof}
For the case where $\Lambda = \Z^n$, we can choose $B = [\e_1,\dots, \e_n]$. For Lenstra's algorithm, $||\d||_2 \leq n$, hence the time complexity simply becomes $n^{O(1)}$.\\
After choosing a flatness direction $\d$, if the width of the inner ellipsoid is smaller than $\omega(n)$, we will project into hyperplanes perpendicular to $\d$. According to Lemma \ref{basis}, we compute vectors $\b_1, \dots, \b_{n-1} \in \Z^n$ such that $B = [\b_1, \dots,\b_{n-1}, \d]$ is a basis for the lattice $\Z^n$. We now consider the projection of $Y$ into the hyperplane $\d^T\x = t$,
$$Y_t = \{ \tilde{\x} \in \R^{n-1} : B [\tilde{\x},t] \in Y\}. $$
This step must ensure the set $Y_t$ is also a convex set in class $\mathcal{C}$ to allow for the recursion to work.
Since $\d$ is a flatness direction of $E(A,\a)$, we know that $|\d^T (\x-\a) | \leq \omega(n)\beta(n)$ for every $\x \in E(A,\a)$ and it suffices to choose $t$ in the set 
$$
\{ z + \left\lfloor \d^T \a\right\rfloor : |z| \leq \omega(n)\beta + 1, z \in \Z\},
$$
which has fewer than $2 \omega(n) \beta + 3$ elements. This means that if the algorithm runs to its full extent, the total number of subcases it will have to evaluate is
\begin{equation}
\label{subcases}
\text{\textbf{worst $\#$ of subcases} } = \prod_{i=1}^n (2\omega(i)\beta(i) + 3).
\end{equation}
Heinz and Khachiyan and Porkolab follow Lenstra's original algorithm which uses $\omega(n) = 2^{O(n^2)}$ and $\beta(n) = O(n^{3/2})$, which leads to a total number of subcases
$$
\textbf{ original worst $\#$ of subcases} = \prod_{i=1}^n (2^{O(i^2)}O(i^{3/2} )+ 3) = 2^{O(n^3)}.
$$
Our new algorithm has three important features. First, it applies the SVP algorithm of Micciancio and Voulgaris to obtain a flatness direction in $n^{O(1)}2^{2n}$ time as opposed to the previous $2^{O(n \log n)}$ time of Kannan.  We use the best known flatness constant for ellipsoids, $\omega(n) = n$. And we achieve an $O(n)$-rounding by using $m:= n2^n$ test points for our separation oracle. This is the first time, to our knowledge, that this choice has been made, and we point out the important fact that making this choice improves the exponential coefficient in the final complexity. In our algorithm, the worst case number of subcases reduces to 
\begin{equation}
\textbf{modern $\#$ of subcases} = \prod_{i=1}^n(2(i)(O(i)) + 3) =  2^{n}(n!)^2 \leq
2^{2n\log_2(n)+ n}. 
\end{equation}
\section{Lenstra-type Algorithm}
 
Here we state a modern Lenstra-type algorithm for a class $\mathcal{C}$ of convex sets and we comment on the overall complexity. This algorithm requires that $\mathcal{C}$ be closed under projection into affine hyperplanes in $\R^n$. 
\vspace{-.5cm}
\begin{center}
\line(1,0){350}\\
\end{center}

\vspace{-.3cm}
\noindent\textbf{Input:} A convex set $Y\subset \R^n$ in class $\mathcal{C}$.\\
\textbf{Output:} A point $\x^* \in Y \cap \Z^n$ or confirmation that no such point exists.
\vspace{-.8cm}
\begin{center}
\line(1,0){350}\\
\end{center}
\vspace{-.3cm}
\textbf{PROCEDURE:}
\begin{enumerate}
\setlength{\itemsep}{0pt}
\item \textbf{Bounds:} Determine $R \in \Z_+$ and an $\epsilon > 0$ such that $Y \subseteq B(\0,R)$ and if $Y \cap \Z^n \neq \emptyset$, then $\vol(Y) > \epsilon$.
\item \textbf{Ellipsoid Rounding:} Compute an ellipsoid $E(A,\a)$ for such an $\epsilon$ such that either 
\begin{enumerate}
\item $Y \subseteq E(A,\a)$ and $\vol(E(A,\a)) \leq \epsilon$, or
\item $E(A,\a)$ is a $\beta$-rounding of $Y$. 
\end{enumerate}
If we are in case (a), then no such point exists.\\
Otherwise we proceed as we are in case (b).
\item \textbf{Flatness Direction:} Compute a flatness direction $\d\in \Z^n$ of $E(\frac{1}{\beta^2} A,\a)$. Either
\begin{enumerate}
\item $w(E(\frac{1}{\beta^2} A,\a)) > \omega(n)$, then there exists a point $ \x^* \in E(\frac{1}{\beta^2} A, \a) \subset Y \cap \Z^n$, which we can compute by solving a closest vector problem, or 
\item proceed knowing that $w(E(A,\a)) \leq \omega(n)\beta(n)$.
\end{enumerate}
\item \textbf{Sublattice:} Compute vectors $\b_1, \dots, \b_{n-1} \in \Z^n$ such that \\$B = [\b_1, \dots, \b_{n-1} ,\d]$ is a lattice basis for $\Z^n$.
\item \textbf{Project: } For each $t \in \{ z + \left\lfloor \d^T \a\right\rfloor : |z| \leq \omega(n)\beta(n) + 1, z \in \Z\},$ 
solve the $n-1$ dimensional integer feasibility subproblem on the set
$$
Y_t = \{ \tilde{\x} \in \R^{n-1} : B [\tilde{\x},t] \in Y\}.
$$
\end{enumerate}
\vspace{-1cm}
\begin{center}
\line(1,0){350}\\
\end{center}

Considering Lenstra's algorithm in the form presented here, we find it has time complexity of
$$
\underbrace{(n + \langle R \rangle + \langle \epsilon \rangle)^{O(1)}}_{\text{Ellipsoid Rounding} }
\underbrace{(\text{Shallow Cut} + m \times \text{feasible test})}_{\text{Shallow Cut Oracle}}
\underbrace{n^{O(1)}2^{2n}}_{\mathrm{SVP/CVP}}
\underbrace{n^{O(1)}}_{\text{Sublattice}}
\underbrace{2^{2n\log_2(n) + n}}_{\text{Subcases}},
$$
where we have left the shallow cut and the feasibility test as unknowns since they are dependent on the class $\mathcal{C}$.
\section{Quasiconvex Shallow Cuts}
\label{quasiconvexMinimization}
In this section we will show that the modern Lenstra-type algorithm can be applied to convex sets given by quasiconvex polynomial inequalities. We will begin with a contribution on efficiently encoding polynomials to exploit sparsity. We will then review properties of quasiconvex polynomials that will be useful for making shallow cuts and present our shallow cut oracle.
\subsection{Polynomial Encoding} 
In this paper, we allow our complexity results to vary based on the encoding scheme chosen for the polynomials. Multi-variable polynomials can be presented in a list of the coefficients of all the monomials up to degree $d$, requiring a large storage space. This is typically referred to as a \emph{dense encoding}. Under this scheme, the following remark holds.
\begin{remark}[Remark 2.1 in~\cite{heinz05}]
Let $F \in \Z[\x]$ be a polynomial of total degree $d$ at most with integer coefficients of binary length bounded by $l$. Moreover, let $\hat{\x} \in \mathbb{Q}^n$ be a fixed point with binary encoding size $\langle \hat{\x} \rangle < r$. Then there is an algorithm with time complexity $(l r n)^{O(1)} d^{O(n)}$ and output-complexity $(l+r) (d n)^{O(1)}$ which computes the value of the function $F$ and the gradient $\nabla F$ at the point $\hat{\x} \in \mathbb{Q}^n$.
\end{remark}
\noindent This time-complexity, however, is too pessimistic; for example, it seems to require $n^{O(n)}$ time to evaluate a monomial of degree $d = n$.\\

An alternative is \emph{sparse encoding}, where monomials are listed with their non-zero exponents and their coefficients, allowing for a more concise representation for short polynomials and a more refined time-complexity analysis. Polynomials and their gradients can then be evaluated in $(l r d M n)^{O(1)}$ time, where $M$ is a bound on the number of monomials in the polynomial. This scheme is potentially problematic in Lenstra's algorithm because each subproblem is realized by intersecting our region with a hyperplane, which would cause a loss of sparsity (fill-in). For instance, if the given polynomial is $x_n^d$ and our hyperplane is $x_n = x_1 + \dots + x_{n-1} + 1$, then in the reduced dimension it becomes $(x_1 + \dots + x_{n-1} + 1)^d$. We note that in the algorithm, we never expand these expressions, allowing sparse encoding to continue to be useful. We instead leave the polynomials alone and store coordinate transformation matrices at each step and then compute the coordinates in the original space to input into the polynomials. Gradients are computed via the chain rule. For context, this is discussed in more detail in Remark \ref{injection}.\\
\subsection{Quasiconvex Polynomials}
A function $F\colon \R^n \to \R$ is called \textit{quasiconvex} if all the lower level sets $\{\x \in \R^n \colon F(\x) \leq \alpha\}, \alpha \in \R$ are convex subsets of $\R^n$. Although quasiconvex functions are not necessarily convex, all convex functions are quasiconvex. We follow~\cite{heinz05} for a review on quasiconvex polynomials.
\begin{figure}
\begin{center}
\ifpdf
\includegraphics[trim = 6in 4.5in 6in 4.5in]{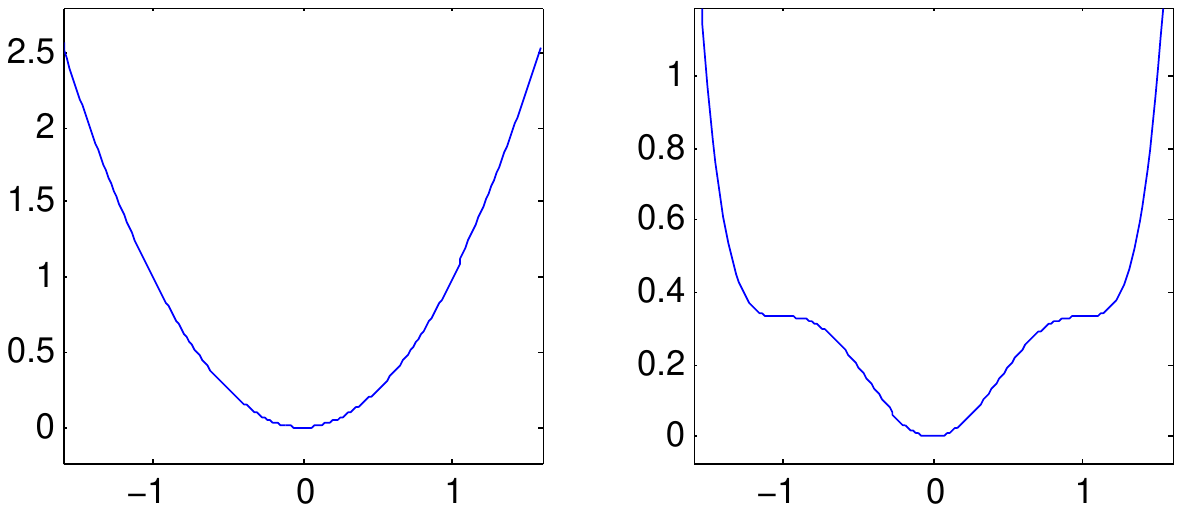}
\else
\includegraphics{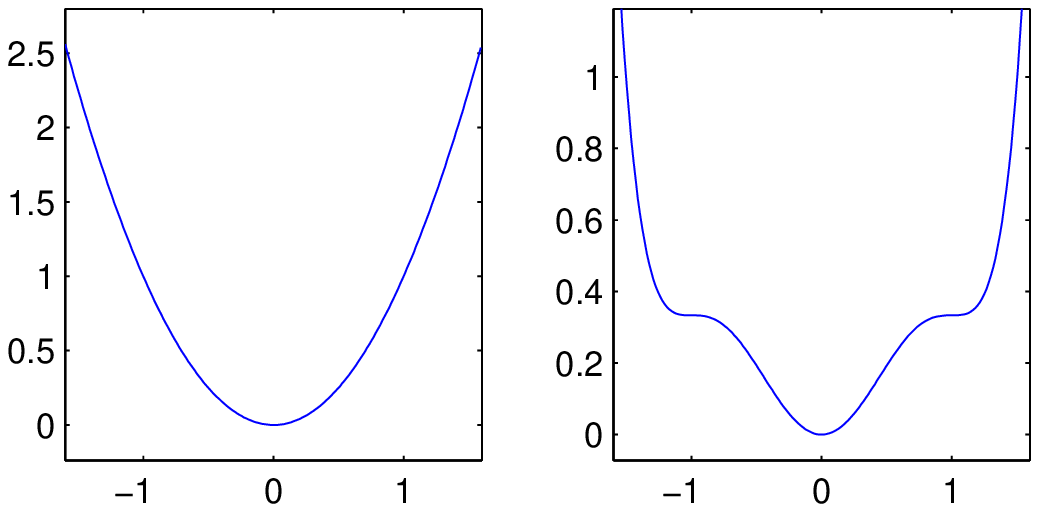}
\fi
\\
\end{center}
\vspace{-.3in} (a) \hspace{2.25in} (b)\hfill \\
\caption{(a) Convex polynomial $ f(x) = x^2$, (b) quasiconvex polynomial $g(x) = x^2 - x^4 + \frac{x^6}{3}$}
\end{figure}
\begin{lemma} [Section 4.1, Remark 1 in~\cite{bank88}]
\label{strict}
Let F $\in \R[\x]$ be a quasiconvex polynomial, $\a\in \R^n$ a fixed point and $\b \in \R^n$, $\b \neq \0$ a fixed vector. If the polynomial $F(\a+ \lambda \b)$ in $\lambda \in \R$ is strongly decreasing (or
constant, respectively), then $F(\x + \lambda \b)$ is strongly decreasing (or constant, respectively) for all $\x \in \R^n$.
\end{lemma}

This lemma does not necessarily hold if the function is not a polynomial. Consider an example from \cite{boyd2004},  
$
f(\x) = \max \{ i \colon x_i \neq 0\}.
$
This is quasiconvex because all the lower level sets are subspaces, for example  
$\{\x
\colon f(\x) < n\} = \{ \x \colon  x_n = 0\}$.
This is a counterexample since $f( (x_1, ..., x_{n-1}, 1)  ) = n $ (a constant), but $f( (x_1, ..., x_{n-1}, 0)  )$  can vary with the remaining input.

  Lemma \ref{strict} can be used to determine if a quasiconvex polynomial is constant.
\begin{corollary}[Corollary 2.3 in~\cite{heinz05}] 
\label{constant}
Let $F \in \R[\x]$ be a quasiconvex polynomial of degree $d$ at most, $\a \in \R^n$ a
point, and let the set $\{\b_1, \dots, \b_n\} \subset \R^n$ be a basis of $\R^n$. If for every $i = 1, \dots, n$, there
are pairwise distinct real numbers $\lambda_{i1}, \dots , \lambda_{id} \in \R$ satisfying $\nabla F ( \a + \lambda_{ij} \b_i ) = \0$
for all $j = 1, \dots , d,$ then the polynomial $F$ is constant.
\end{corollary}
The following lemma is  important for generating shallow cuts.
\begin{lemma}[Lemma 2.4 in~\cite{heinz05}]
Let $F \in \R[\x]$ be a quasiconvex polynomial and let $\y \in \R^n$ be a fixed point.
If $F(\y) \geq 0$ and $\nabla F(\y) \neq \0$, for every other $\x \in \R^n$ that satisfies $F(\x) < 0$, we have that 
$$
\nabla F(\y) \cdot \x \leq \nabla F(\y) \cdot \y .
$$
\end{lemma}
As mentioned before, the class of convex sets used must be closed under intersections with affine hyperplanes.  We will  also require that an ellipsoid bound $\x^T A_0 \x < R$ reduce to similar ellipsoid bound.
\begin{remark}[Within the proof of Theorem 4.2 in~\cite{heinz05}]
\label{polyreduction}
Let $F_0, \dots, F_{s} \in \Z[\x]$ be quasiconvex polynomials, $R \in \Z$, $A_0 \in \Z^{n \times n}$ a positive definite matrix, 
and $F_{s+1} \in \Z[\x]$ a polynomial defined by $F_{s+1}(\x) := -R + \x^T A_0 \x$, for $\x\in \R^n$. 
Moreover, let the binary length of the coefficients be bounded by $l$, let $d$ be an upper bound for the degree of the polynomials. 
Let $B \in \Z^{n\times n}$ be nonsingular, $t \in \Z$, with entries of $B$ and $t$ of binary length at most $l(d n)^{O(1)}$. Let 
$$
Y = \{ \x \in \R^n : F_i(\x) < 0 , i = 1, \dots, s+1\}
$$
and let
$$
Y_t := \{\tilde{\x} \in \R^{n-1} : B [\tilde{\x},t] \in Y\}.
$$
%
Consider the set $Y_t$ and the new coordinates $\tilde{x}_1, \dots, \tilde{x}_{n-1}$ induced by $\x = B [\tilde{\x}, t]$, 
fixing the last coordinate $\tilde{\x}_n = t$ and rewriting the quasiconvex polynomials in terms of the new coordinates. 
The maximum binary length of all coefficients belonging to the new polynomials 
$\tilde{F}_0, \dots, \tilde{F}_{s+1} \in \Z[\tilde{x}_1, \dots, \tilde{x}_{n-1} ]$, is $l (d n)^{O(1)}$. \\
Furthermore, all new polynomials are quasiconvex since the transformation is linear and 
$F_{s+1}$ preserves its form for a new suitable $\tilde{A}_0$. 
The degree bound $d$ and the number $s$ of polynomials remain unchanged, but the number of coordinates reduces by one. 
\end{remark}
\begin{remark}
\label{injection}
Following the notation of Remark \ref{polyreduction}, we will explain here how we evaluate the polynomials and their gradients under the sparse encoding scheme. Suppose $k+1$ of such coordinate transformations $B^n, \dots , B^{n-k} \in \Z^{n\times n}$ are done to produce the variable $\tilde{\x}^{n-k} \in \Z^{n-k-1}$. Each $B^{n-i}$ is a block diagonal matrix where the last block is an identity matrix of size $i$. In each transformation $B^{n-i}$, we are restricting the last variable to be $t_i$. A polynomial $F$ transformed into the new coordinates we will denote as $\tilde{F}^{n-k}$. For a given $\tilde{\x}^{n-k} \in \Z^{n-k}$, we can compute $\tilde{F}^{n-k}(\tilde{\x}^{n-k})$ as
$$
\tilde{F}^{n-k}(\tilde{\x}^{n-k}) = F( \x)
$$
where
$$
\x = B^n B^{n-1} \cdots B^{n-k} \begin{bmatrix} \tilde{\x}^{n-k}\\ t_{n-k} \\ \vdots \\ t_n \end{bmatrix}.
$$
With the product $C^{n-k} = B^n\cdots B^{n-k}$ computed ahead of time, a depth first search allows us to store at most $n$ of these products at any given time. The partial derivatives of $\tilde{F}^{n-k}$ then have a simple representation as
$$
\frac{\partial \tilde{F}^{n-k}}{\partial \tilde{x}^{n-k}_i} = \nabla F(\x) \cdot\frac{\partial \x}{\partial \tilde{x}^{n-k}_i} = \nabla F(\x)\cdot C^{n-k}_i
$$
where $C^{n-k}_i$ is the $i^{\text{th}}$ column of $C^{n-k}$. 
\end{remark}
\subsection{Shallow Cuts}
The main result of \cite{heinz05} is derived from  Heinz's shallow cut separation oracle for quasiconvex polynomials. The following is an adaptation of Heinz's proof to allow for stronger ellipsoid roundings. Specifically, we show that his calculation for a basis direction to admit a shallow cut generalizes to having any direction admit a shallow cut.
Recall that we are solving the feasibility problem over the set 
$$
Y = \{\x\in \R^n : F_i(\x) < 0 \text{ for } i = 0, 1, \dots, s \}
$$
where all the $F_i's$ are quasiconvex polynomials with integer coefficients. Consider an ellipsoid $E(A,\a)$ and let $\{\b_1, \dots, \b_n\} \subset \R^n$ be an orthogonal basis of $\R^n$ according to the matrix~$A$ (where the inner product is given by $\langle \x, \y \rangle_{A} = \x^T A \y$). Define the affine map $\tau: \R^n \to \R^n$ such that
\begin{equation}
\tau(\x) := B^T(\x - \a)
\end{equation}
where
\begin{equation}
B := \left( \frac{\b_1}{||\b_1||_A}, \dots, \frac{\b_n}{||\b_n||_A} \right) \in \R^{n \times n}.
\end{equation}
Thus $\x^T \x \leq 1$ if and only if $\tau^{-1}(\x) \in E(A,\a)$.
\begin{theorem}[Shallow Cut Separation Oracle]
\label{shallow}
Let $\hat{c} > 1$ and let $m\colon \N \to \N$, such that $2n \leq m(n) \leq \hat{c}^n$. Then there exists a function $\beta\colon \N \to \R$ where $\beta(n) = O(\frac{n^{3/2}}{\log(m/n)})$ and an algorithm with the following input:\\
\indent $\mathrm{(I_1)}$ sparsely encoded quasiconvex polynomials $F_0, \dots, F_{s+1} \in \Z[\x]$ of total degree $d$, at most $M$ monomials in each, and whose coefficients' binary encoding lengths are bounded by $l$,\\
\indent $\mathrm{(I_2)}$ an ellipsoid $E(A, \a)$ containing $Y$ as defined in \eqref{Y}, where the binary encoding length of the columns of $A$ and of $\a$ are bounded by $r$,\\
and outputs one of the following answers:
\begin{enumerate}[\quad\rm 1.]
\setlength{\itemsep}{0pt}
\item confirmation that the ellipsoid $E(A,\a)$ is a $\beta$-rounding of $Y$, or
\item a vector $\c \in \Q^n, \c \neq \0$, with the property 
\begin{equation}
Y \subset \left\{\x \in \R^n : \c^T \x \leq \c^T\a+ \frac{1}{n+1} ||\c||_A \right\}.
\end{equation}
\end{enumerate}
This algorithm runs in time-complexity $s(l n r m M)^{O(1)}$
and with\\ output-complexity $(l+r)(d n)^{O(1)}$. 
\end{theorem}
\begin{proof}
First compute an orthogonal basis $\{\b_1, \dots, \b_n\}$ according to $A$. Let $\sigma > 1$ (Heinz used $\sigma= 3/2$). Next construct a polytope approximating $S^{n-1}(1)$ according to Theorem \ref{rounding} using $m$ vertices and let $V\subset \Z^n$ denote the set of non-normalized vertices. 
For every $\v \in V $, define
$$
\b_{\v} := \sum_{i=1}^n \frac{v_i \b_i}{||\v||_2 ||\b_i||_A},
$$
$$
\x_{\v} := \a + \frac{1}{ (n+\sigma)} \b_{\v}.
$$
Since we cannot compute $\b_{\v}$ and $\x_{\v}$ exactly due to square roots from the norms, we will approximate the square roots. This can be done with any root finding technique using fixed point arithmetic, for example, the Newton-Raphson method has quadratic convergence and will find the desired approximation within polynomial time. For a reference on numerical methods, see~\cite{burden}. We note that 
$$
||\v||_2 ||\b_i||_A = \sqrt{\v^T \v} \sqrt{\b_i^T A \b_i} = \sqrt{ \v^T \v\b_i^T A \b_i}.
$$ 
Thus we will approximate the reciprocal of that square root within an accuracy of $\delta = \epsilon/(||A^{1/2}||_\infty ||\b_{\max}||_2 \sqrt{n})$ where $||A||_\infty$ is the maximum row sum of $A$ and $\b_{\max} = \mathrm{argmax}\{||\b_i||_2 : i=1, \dots, n\}$. 
Let $0 < \delta_{\v,i}< \delta $ be the exact error on each approximation. 
Let $\tilde{\b}_{\v} \in \mathbb{Q}^n$ and $\tilde{\x}_{\v}\in \mathbb{Q}^n$ denote the rational approximations of $\b_{\v}$ and $\x_{\v}$ respectively. 

Since $V$ is symmetric across all axes, it suffices to store only the vertices in the first orthant. This exponentially reduces the number of root approximations necessary.
\textit{\\ Case 1: ($\beta$-rounding)} Suppose $\tilde{\x}_{\v} \in Y$ for all $\v \in V$. We will show that $E(A, \a)$ is a $\beta$-rounding of $Y$.\\
 \\
First observe that 
\begin{align*}
\tau( \a + A\tilde{\b}_\v) &= B^T A\tilde{\b}_{\v}\\
 &= \sum_{i=1}^n \v_i \frac{\b_i^T A \b_i}{||\b_i||_A} \left(\frac{1}{||\v||_2 ||\b_i||_A} + \delta_{\v,i}\right)\\
 &= \sum_{i=1}^n \e_i \v_i (\frac{1}{||\v||_2 } + \delta_{\v,i} ||\b_i||_A),
\end{align*}
where $\e_i\in \R^n$ is the $i^{\text{th}}$ unit vector. Hence,  we have
\begin{align*}
\left|\left| \frac{\v}{ ||\v||_2} - \tau(\a + A\tilde{\b}_\v)\right|\right|_\infty &=\max_{i=1, \dots, n} \v_i \delta_{\v,i} ||\b_i||_A\\
&\leq \max_{i=1, \dots, n} \delta_{\v,i} ||\b_i||_A\\
&\leq \delta ||A^{1/2}||_\infty ||\b_{\max}||_2\\
&\leq \epsilon/\sqrt{n}.
\end{align*}
  Since  $||\cdot||_\infty \leq \epsilon/\sqrt{n}$ implies $||\cdot||_2 \leq \epsilon$, we see that $\{ B^T A\tilde{\b}_{\v} : \v \in V\}$ is an $\epsilon$-approximation of $V$ (after normalizing to the unit sphere). Therefore by Corollary~\ref{rationalapprox}, there exists a $\hat{\beta} = O(\sqrt{n/\log(m/n) })$ such that $\conv(K) \subset E(\frac{1}{\hat{\beta}^2} I, \0)$. Letting $\beta = \hat{\beta}(n+ \sigma)$, it follows that
$$
E(\tfrac{1}{\beta^2}A, \a) = \tau^{-1}(E(\tfrac{1}{\beta^2}I,\0)) \subset \conv(\{ \tilde{\x}_{\v} : \v \in V\}) \subset Y \subset E(A,\a).
$$ 
\indent \textit{Case 2: (Shallow cut)} Suppose that $\tilde{\x}_{\bar{\v}} \notin Y$ for some $\bar{\v} \in V$. We will show that there exists the desired hyperplane.\\
\indent For some $F \in \{F_0, \dots, F_{s}\}$, we know that $F(\tilde{\x}_{\bar{\v}}) \geq 0$.\\
\indent \textit{Case 2.1:} $F(\a) < 0$. We will find a point near $\tilde{\x}_{\bar{\v}}$ with non-zero gradient. \\
Pick scalars $\lambda_1, \dots, \lambda_d$, such that 
\begin{equation}
\frac{n+1}{n+\sigma} < \lambda_1 < \dots < \lambda_d < \frac{1}{||\tilde{\b}_{\bar{\v}}||_A}.
\end{equation}
\indent Since the inequalities $F(\a) < 0 $ and $F(\tilde{\x}_{\bar{\v}}) \geq 0$ are valid, the polynomial $F(\a + \frac{1}{n+1}\lambda \tilde{\b}_{\bar{\v}}) $ is of degree $d$ at most and not constant with respect to $\lambda$. Therefore, for some $1\leq k \leq d$, we may choose a point $\y \in \R^n$ satisfying
$$
\y = \a + \frac{1}{n+1}\lambda_k A \tilde{\b}_{\bar{\v}} \text{ and } \nabla F(\hat{\y}) \neq 0.
$$
Define $\c := \nabla F(\y)^T$. Note that $\y \notin Y$ since $\tilde{\x}_{\v}$ is a convex combination of $\a$ and $\y$. Thus
for all $\x \in Y$,
$$
\c^T \x \leq \c^T \y = \c^T\a + \frac{1}{n+1} \lambda_k \c^T A \tilde{\b}_{\bar{\v}} \leq \c^T\a + \frac{1}{n+1} \frac{\c^T A \tilde{\b}_{\bar{\v}}}{||\tilde{\b}_{\bar{\v}}||_A} \leq \c^T\a + \frac{1}{n+1} ||\c||_A.
$$
The last inequality comes from the Cauchy-Schwarz inequality for the scalar product generated by the matrix $A$.\\
\indent \textit{Case 2.2:} $F(\a) \geq 0$ (i.e., $\a \notin Y$). We can now completely ignore the fact that $\tilde{\x}_{\tilde{\v}} \notin Y$ and instead just find a unit direction point with non-zero gradient.  If $\nabla F (\a) \neq \0$, then we can simply use $\c = \nabla F(\a)$.  Otherwise, we may need to examine up to $nd$ points not in $Y$.  

For this reason, pick scalars $\lambda_{i1}, \dots, \lambda_{id}$ for each $i = 1, \dots, n$ such that 
$$
0 < \lambda_{i1} < \dots < \lambda_{id} < \frac{1}{||\b_i||_A}
$$
once again using the Newton-Raphson method to approximate the roots, but this time the same precision is not required. 

 Define the set
$$
C = \left\{ \a \pm\gamma \lambda_{ij} A \b_i \colon j = 1, \dots, n\right\},
$$
where $0 < \gamma \leq \frac{1}{n+1}$. By Lemma \ref{constant}, if $\nabla F(\y) = \0$ for all $\y \in C$, then $F$ is constant. If so, any point $\c\in \Q^n$ will suffice to output.\\

Otherwise, we will show that there exists a point in $C\setminus Y$ with non-zero gradient. 

We remark here that we only need to choose $\gamma = \frac{1}{n+1}$ to generate a shallow cut, but choosing a smaller $\gamma$ will allow for a deeper cut and therefore a faster convergence of the shallow cut ellipsoid method.  Unfortunately, this does not improve the theoretical complexity of the algorithm.  See \cite[chapter 4]{schriver88} for more about the complexity of various cuts.\\
For every $i = 1, \dots, n$, define the finite subsets 
\begin{align*}
C^+_i &:= \left\{ \a +\gamma\lambda_{ij} A \b_i \colon j = 1, \dots, n\right\},\\
C^-_{i} &:= \left\{ \a -\gamma \lambda_{ij} A \b_i \colon j = 1, \dots, n\right\}.
\end{align*}
For every $i=1, \dots, n$, since $\a$ is a convex combination of points in $C^+_i$ and $C^-_{i}$, at least one of the sets $C^+_i \cap Y$ or $C^-_{i}\cap Y$ is empty. 
Thus there exists a point $\y \in C$ such that, similar to Case 2.1, 
$$
F(\y) \geq 0 \text{ and } \nabla F(\y) \neq \0.
$$ 
Define $\c = \nabla F(\y)$. The remaining calculation is similar to the one above.
\end{proof}
\begin{figure}
\begin{center}
\ifpdf
\input{case2RoundingProofpdf.tex}
\else
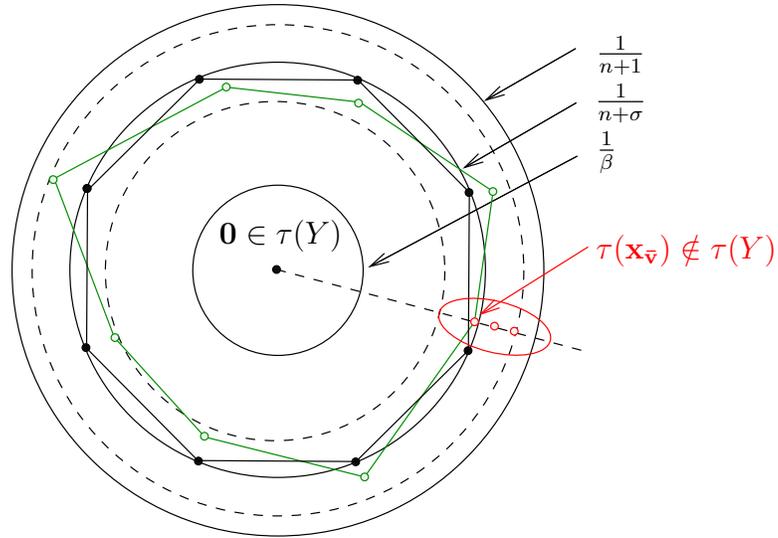
\fi\\
\end{center}
\caption{Case 2.1: If $F(\a) < 0$ and $F(\x_{\bar{\v}}) \geq 0$, then $F$ restricted to the line through those two points is not constant. One of the test points on that line must have a non-zero gradient and not lie within $Y$.}
\end{figure}
\begin{figure}
\begin{center}
\ifpdf
\input{case22RoundingProofpdf.tex}
\else
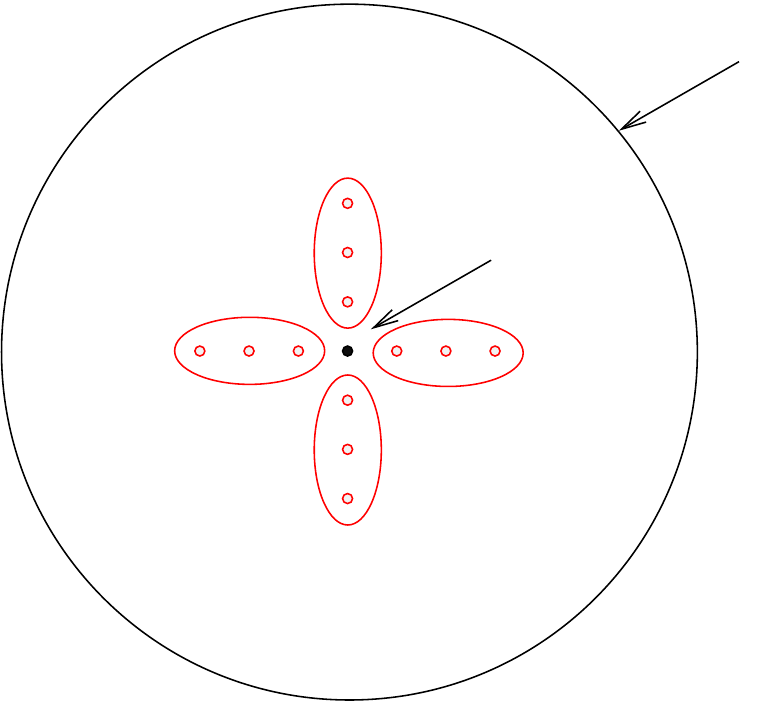
\fi\\
\end{center}
\caption{Case 2.2: If $F(\a) \geq 0$ we search for a point near $\a$ in a unit direction that has gradient non-zero and is not in $Y$.}
\end{figure}

This proof differs from Heinz's in three ways.  Most obviously, we generalize the test points for a rounding based on any chosen $m\geq 2n$.  Secondly, we generalize Heinz's cuts to show that any point $\y$ with $\tau(\y) \in B(\frac{1}{n+1},\0)$ and $\nabla F(\y) \neq \0$ generates a shallow cut.  Lastly, in case 2.2, Heinz's proof requires $\lambda_{ij}$ bounded from below as the previous $\lambda_{i}$ coefficients were.  As we show, that is not necessary, and is even disadvantageous.  Allowing a different $\gamma$ for a deeper cut is a new idea.  

\begin{corollary}
\label{shallowCorollary}
Let $\hat{c} > 1$ and let $m\colon \N \to \N$, such that $2n \leq m(n) \leq \hat{c}^n$. Then there exists a function $\beta\colon \N \to \R$ where $\beta(n) = O(\frac{n^{3/2}}{\log(m/n)})$ and an algorithm with the following input:\\
\indent $\mathrm{(I_1)}$ sparsely encoded quasiconvex polynomials $F_0, \dots, F_{s} \in \Z[\x]$ of total degree $d$ and of at most $M$ monomials,\\
\indent $\mathrm{(I_2)}$ an integer $R$ and a positive definite matrix $A_0 \in \Z^{n\times n}$ and define $F_{s+1}(\x) := -R + \x^T A_0 \x$. Let $l$ be a bound on the maximum binary encoding length of the coefficients of $F_0, \dots , F_{s+1}$,\\
\indent $\mathrm{(I_3)}$ a positive number $\epsilon \in \mathbb{Q},$\\
which outputs a positive definite matrix $A \in \mathbb{Q}^{n\times n}$ and a point $\a\in \mathbb{Q}^n$ such that one of the following holds:
\begin{enumerate}[\quad\rm 1.]
\setlength{\itemsep}{0pt}
\item $Y \subseteq E(A,\a)$ and $\vol(E(A,\a)) \leq \epsilon$, or
\item the ellipsoid $E(A, \a)$ is a $\beta$-rounding of $Y$.
\end{enumerate}This algorithm runs in time-complexity $s (l n m d M \langle \epsilon \rangle)^{O(1)}$ 
and with \\
output-complexity $(l + \langle \epsilon \rangle) (d n)^{O(1)}$. 
\end{corollary}
\begin{proof}
The proof is similar to~\cite[Corollary 3.4]{heinz05}.
\end{proof}
\section{Lenstra-type Algorithm}
Now that we have a shallow-cut separation oracle, our Lenstra-type algorithm is almost ready to be applied to quasiconvex integer minimization.  Before arriving at the proof of main theorem, we state the following lemmas related to the convergence of the shallow-cut ellipsoid method and conclude that we can solve the feasibility problem.
\begin{lemma}[Lemma 4.1 in \cite{heinz05}]
\label{minvolume}
Let $F_0, \dots, F_{s+1} \in \Z[\x]$ be polynomials, $R \in \Z$ an integer, $A_0 \in \Z^{n \times n}$ a positive definite matrix, and 
$F_{s+1} \in \Z[\x]$ a polynomial defined by $F_{s+1}(\x) := -R + \x^T A_0 \x$ for $\x \in \R^n$. 
Moreover, let the binary encoding length of the coefficients be bounded above by $l$, let $d$ be an upper bound for the degree of the polynomials, and let the set $Y$ contain an integer point $\hat{\x} \in \Z^n$. 
Then there is a positive rational number $\epsilon \in \mathbb{Q}$ which bounds the volume 
$0 < \epsilon < \vol(Y)$ such that its binary length $\langle \epsilon\rangle $ is in the class $l (d n)^{O(1)}$.
\end{lemma}
\begin{lemma}[p.27 in \cite{bank97}]
\label{R}
If a minimum point exists to problem (\ref{minimization}), then there exists a ball of radius $R_* \in \Z$ containing such a point, where the binary length of $R_*$ is $l d^{O(n)}$.
\end{lemma}
\noindent The following theorem solves the feasibility problem.
\begin{theorem} 
\label{feasible}
Let $F_0, \dots, F_s \in \Z[\x]$ be quasiconvex polynomials, $R >0$ an integer, $A_0 \in \Z^{n\times n}$ a positive definite matrix, and 
$F_{s+1} \in \Z[\x]$ a polynomial defined by $F_{s+1}(\x) = -R + \x^T A_0 \x$ for $\x \in \R^n$. 
Let $d$ be an upper bound for the degree of the polynomials $F_0, \dots, F_{s+1}$, presented as a sparse list of monomials with at most $M$ monomials, and let the binary length of the coefficients be bounded by $l$. Moreover, consider the set
\begin{equation}
Y := \left\{ \x \in \R^n : F_i(\x) < 0 , i = 0,1,\dots,s+1 \right\}
\end{equation}
Then there is an algorithm with 
time-complexity $s (d M l)^{O(1)} 2^{2n\log_2(n) + O(n)} $
which computes a point $\x^* \in Y \cap \Z^n$ or confirms that no such point exists. 
\end{theorem}
\begin{proof}  This is a direct result of the complexity analysis done in section 4 in conjunction with Lemma \ref{minvolume}, Lemma \ref{R}, and Corollary \ref{shallowCorollary}.
\end{proof}
From the approach of this method, due to the recursion, it is unlikely possible to obtain a better time-complexity than $2^{O(n\log n)}$ in terms of the dimension. An open problem is then to find an algorithm that is single exponential in the dimension.\\ 
We now provide an outline for the proof of Theorem \ref{main}, which follows from the same reasoning as~\cite[Theorem 5.1]{heinz05}.
\begin{proof}[Proof of Theorem~\ref{main}]
Again, from \cite[p. 27]{bank97} that if a minimum point $\x^*$ exists, then there exists a ball of radius $R_* \in \Z$ containing such a point, where the binary length of $R_*$ is $r = l d^{O(n)}$.  Therefore, $F(\x^*)$ has binary length bounded by $(lrdMn)^{O(1)} = (lMn)^{O(1)} d^{O(n)}.$  We define the polynomial $F_{s+1} \in \Z[\x]$ by $F_{s+1}(\x) \colon = -R_*^2 + \x^T \x$ for $\x \in \R^n$, making the problem bounded. To solve the optimization problem, we compute the smallest integer $z^* \in \Z$ such that 
$$
\left\{ \x \in \Z^n \colon F_0(\x) - z^* <0, \text{ and } F_i(\x) < 0 \text{ for } i = 1, \dots, s+1\right\} \neq \emptyset.
$$
We then apply binary search on $z^*$ within binary length of $(lMn)^{O(1)}d^{O(n)}$, testing for integer points using Theorem \ref{feasible}, to find an optimal $z^*$ and obtain the desired time-complexity because of the bound given by \cite{bank97}.
\end{proof}

\section*{Acknowledgments}
The first author was supported by grant DMS-0636297, and both authors
were supported by grant DMS-0914873 of the National Science
Foundation.
\bibliography{../bib/research}{}

\begin{thebibliography}{10}

\bibitem{ajtai01}
M.~Ajtai, R.~Kumar, and D.~Sivakumar.
\newblock A sieve algorithm for the shortest lattice vector problem.
\newblock In {\em STOC '01: Proceedings of the Thirty-third Annual ACM
  Symposium on Theory of Computing}, pages 601--610, New York, NY, USA, 2001.
  ACM.

\bibitem{anstre99}
K.~M. Anstreicher.
\newblock Ellipsoidal approximations of convex sets based on the volumetric
  barrier.
\newblock {\em Mathematics of Operations Research}, 24(1):193--203, 1999.

\bibitem{anstre02}
K.~M. Anstreicher.
\newblock Improved complexity for maximum volume inscribed ellipsoids.
\newblock {\em SIAM Journal on Optimization}, 13(2):309--320, 2002.

\bibitem{bana93}
W.~Banaszczyk.
\newblock New bounds in some transference theorems in the geometry of numbers.
\newblock {\em Mathematische Annalen}, 296:625--635, 1993.

\bibitem{bana99}
W.~Banaszczyk, A.~E. Litvak, A.~Pajor, and S.~J. Szarek.
\newblock The flatness theorem for nonsymmetric convex bodies via the local
  theory of {B}anach spaces.
\newblock {\em Mathematics of Operations Research}, 24(3):728--750, 1999.

\bibitem{bank97}
B.~Bank.
\newblock Optimization and real equation solving.
\newblock In {\em II Escuela de Matem\'atica Aplicada (25 al 29 de agosto de
  1997): Notas de los Cursos}. Universidad de Buenos Aires, 1997.

\bibitem{bank88}
B.~Bank and R.~Mandel.
\newblock {\em Parametric Integer Optimization}.
\newblock Akademie-Verlag, Berlin, 1988.

\bibitem{bert}
D.~Bertsimas and R.~Weismantel.
\newblock {\em Optimization over Integers.}
\newblock Dynamic Ideas, Belmont, MA, May 2005.

\bibitem{boyd2004}
S.~Boyd and L.~Vandenberghe.
\newblock {\em Convex Optimization}.
\newblock Cambridge University Press, New York, NY, USA, 2004.

\bibitem{burden}
R.~L. Burden and J.~D. Faires.
\newblock {\em Numerical Analysis}.
\newblock Thomson Brooks/Cole, 8th edition, 2005.

\bibitem{dinur03}
I.~Dinur, G.~Kindler, R.~Raz, and S.~Safra.
\newblock {Approximating CVP to within almost-polynomial factors is NP-Hard}.
\newblock {\em Combinatorica}, 23:205--243, April 2003.

\bibitem{50years}
F.~Eisenbrand.
\newblock Integer programming and algorithmic geometry of numbers.
\newblock In M.~J{\"u}nger, T.~Liebling, D.~Naddef, W.~Pulleyblank, G.~Reinelt,
  G.~Rinaldi, and L.~Wolsey, editors, {\em 50 Years of Integer Programming
  1958--2008}. Springer-Verlag, 2010.

\bibitem{schriver88}
M.~Gr{\"o}tschel, L.~Lov{\'a}sz, and A.~Schrijver.
\newblock {\em Geometric Algorithms and Combinatorial Optimization}.
\newblock Springer-Verlag Berlin Heidelberg, 1988.

\bibitem{hanrot07}
G.~Hanrot and D.~Stehl\'{e}.
\newblock {Improved analysis of Kannan's shortest lattice vector algorithm}.
\newblock pages 170--186, 2007.

\bibitem{heinz05}
S.~Heinz.
\newblock Complexity of integer quasiconvex polynomial optimization.
\newblock {\em Journal of Complexity}, 21(4):543--556, 2005.

\bibitem{helfrich85}
B.~Helfrich.
\newblock {Algorithms to construct Minkowski reduced and Hermite reduced
  lattice bases}.
\newblock {\em Theor. Comput. Sci.}, 41:125--139, December 1985.

\bibitem{henk97}
M.~Henk.
\newblock Note on shortest and nearest lattice vectors.
\newblock {\em Information Processing Letters}, 61:183--188, 1997.

\bibitem{john48}
F.~John.
\newblock Extremum problems with inequalities as subsidiary conditions.
\newblock {\em Studies and Essays}, Courant Anniversary Volume:187--204, 1948.

\bibitem{kannan83}
R.~Kannan.
\newblock Improved algorithms for integer programming and related lattice
  problems.
\newblock In {\em STOC '83: Proceedings of the Fifteenth Annual ACM Symposium
  on Theory of Computing}, pages 193--206, New York, NY, USA, 1983. ACM.

\bibitem{kannaninf}
R.~Kannan.
\newblock Minkowski's convex body theorem and integer programming.
\newblock {\em Mathematics of Operations Research}, 12(3):415--440, 1987.

\bibitem{kannan79}
R.~Kannan and A.~Bachem.
\newblock {Polynomial algorithms for computing the Smith and Hermite normal
  forms of an integer matrix}.
\newblock {\em SIAM Journal on Computing}, 8(4):499--507, 1979.

\bibitem{khach96}
L.~Khachiyan.
\newblock Rounding of polytopes in the real number model of computation.
\newblock {\em Mathematics of Operations Research}, 21(2):307--320, 1996.

\bibitem{khach2000}
L.~Khachiyan and L.~Porkolab.
\newblock Integer optimization on convex semialgebraic sets.
\newblock {\em Discrete {\&} Computational Geometry}, 23(2):207--224, 2000.

\bibitem{khin48}
A.~Khinchin.
\newblock {A quantitative formulation of Kronecker's theory of approximation
  (in russian)}.
\newblock {\em Izvestiya Akademii Nauk SSR Seriya Matematika}, 12:113--122,
  1948.

\bibitem{kochol94}
M.~Kochol.
\newblock Constructive approximation of a ball by polytopes.
\newblock {\em Mathematica Slovaca}, 44(1):99--105, 1994.

\bibitem{kochol04}
M.~Kochol.
\newblock A note on approximation of a ball by polytopes.
\newblock {\em Discrete Optimization}, 1(2):229--231, 2004.

\bibitem{kumar05}
P.~Kumar and E.~A. Y{\i}ld{\i}r{\i}m.
\newblock Minimum-volume enclosing ellipsoids and core sets.
\newblock {\em Journal of Optimization Theory and Applications}, 126:1--21,
  2005.

\bibitem{lenstra83}
H.~W. Lenstra, Jr.
\newblock Integer programming with a fixed number of variables.
\newblock {\em Mathematics of Operations Research}, 8:538--548, 1983.

\bibitem{miccisvp01}
D.~Micciancio.
\newblock The shortest vector problem is {NP}-hard to approximate to within
  some constant.
\newblock {\em SIAM Journal on Computing}, 30(6):2008--2035, Mar. 2001.
\newblock Preliminary version in FOCS 1998.

\bibitem{micciancio10}
D.~Micciancio and P.~Voulgaris.
\newblock A deterministic single exponential time algorithm for most lattice
  problems based on voronoi cell computations.
\newblock pages 351--358, 2010.

\bibitem{micci092}
D.~Micciancio and P.~Voulgaris.
\newblock Faster exponential time algorithms for the shortest vector problem.
\newblock In {\em Proceedings of SODA}. ACM/SIAM, Jan 2010.

\bibitem{nesterov07}
{\relax{Yu}}.~Nesterov.
\newblock Rounding of convex sets and efficient gradient methods for linear
  programming problems.
\newblock {\em Optimization Methods Software}, 23(1):109--128, 2008.

\bibitem{nguyen}
P.~Q. Nguyen and T.~Vidick.
\newblock Sieve algorithms for the shortest vector problem are practical.
\newblock {\em Journal of Mathematical Cryptology}, 2(2), 2008.

\bibitem{sto96}
A.~Storjohann and G.~Labahn.
\newblock Asymptotically fast computation of {H}ermite normal forms of integer
  matrices.
\newblock In {\em Proceedings of the 1996 International Symposium on Symbolic
  and Algebraic Computation}, pages 259--266. ACM Press, 1996.

\bibitem{todd07}
M.~J. Todd and E.~A. Y{\i}ld{\i}r{\i}m.
\newblock {On Khachiyan's algorithm for the computation of minimum-volume
  enclosing ellipsoids}.
\newblock {\em Discrete Applied Mathematics}, 155(13):1731--1744, 2007.

\end{thebibliography}
\bibliographystyle{abbrv}
{\small
Robert Hildebrand: Department of Mathematics, University of California, Davis, One Shields Avenue, Davis, CA, 95616, USA\\
\textit{E-mail address:} {\tt{rhildebrand@math.ucdavis.edu}}\\

\noindent Matthias K\"oppe: Department of Mathematics, University of California, Davis, One Shields Avenue, Davis, CA, 95616, USA\\
\textit{E-mail address:} {\tt{mkoeppe@math.ucdavis.edu}}\\
}
\end{document}